\documentclass[11pt]{amsart}

\usepackage{amsmath}
\usepackage{fullpage}
\usepackage{xspace}
\usepackage[psamsfonts]{amssymb}
\usepackage[spanish]{babel}
\usepackage[utf8]{inputenc}
\usepackage{graphicx,color}
\usepackage[curve]{xypic} 
\usepackage{hyperref}
\usepackage{graphicx}

\usepackage{amsmath}%
\usepackage{amsthm}%
\usepackage{amscd}
\usepackage{amsfonts}%
\usepackage{amssymb}%
\usepackage{graphicx}

\usepackage{mathrsfs}

\usepackage{tikz}
\usetikzlibrary{matrix,arrows}

\newtheorem{theorem}{Teorema}

\theoremstyle{remark}

\numberwithin{equation}{section}

\newcommand{\Ccal}{\mathscr{C}}

\newcommand{\Z}{\mathbb{Z}}

\newcommand{\F}{\mathbb{F}}
\newcommand{\Q}{\mathbb{Q}}

\newcommand{\Aut}{\mathrm{Aut}}

  \DeclareFontFamily{U}{wncy}{}
    \DeclareFontShape{U}{wncy}{m}{n}{<->wncyr10}{}
    \DeclareSymbolFont{mcy}{U}{wncy}{m}{n}
    \DeclareMathSymbol{\Sha}{\mathord}{mcy}{"58}

\begin{document}
\title{Sobre los teoremas de Shafarevich y Siegel}

\author{Héctor Pastén}
\address{ Departamento de Matem\'aticas\newline
\indent Pontificia Universidad Cat\'olica de Chile, Facultad de Matem\'aticas\newline
\indent 4860 Av. Vicu\~na Mackenna,  Macul, RM, Chile}
\email[H. Pastén]{hector.pasten@uc.cl}%

\thanks{Financiado por ANID Fondecyt Regular 1230507 de Chile.}
\date{\today}

\begin{abstract} Presentaremos una nueva demostración del teorema de Shafarevich sobre finitud de curvas elípticas con buena reducción fuera de un conjunto finito de primos dado. Esto da un nuevo punto de entrada a teoremas fundamentales de finitud diofantina tales como el teorema de Siegel sobre la ecuación $S$-unidad. Nuestro argumento está libre de aproximación diofantina o teoría de trascendencia, y se acerca más a las ideas de Faltings en su demostración de la conjetura de Mordell. 
\end{abstract}

\maketitle



En lo que sigue $K$ será un campo de números fijo. Nuestro propósito es dar una nueva demostración del siguiente importante teorema de Shafarevich sin utilizar aproximación diofantina:

\begin{theorem}[Shafarevich]\label{TeoShafarevich} Sea $S$ un conjunto finito de lugares de $K$. Las curvas elípticas sobre $K$ con buena reducción fuera de $S$ forman sólo finitas clases de $K$-isomorfismo.
\end{theorem}

\begin{proof} Incluya en $S$ los lugares sobre $2$, $3$ e $\infty$. Basta demostrar finitud de las clases de $K^{\rm alg}$-isomorfismo de las curvas elípticas $E$ sobre $K$ con buena reducción fuera de $S$. Esto debido a la finitud de $H^1(G_{K,S} , \Aut (E_{K^{\rm alg}}))$ con $E_{K^{\rm alg}}$ como grupo algebraico (ver por ejemplo el Corolario 4.15 del Capítulo 1 de \cite{Milne})  lo que no usa aproximación diofantina. 

Desde ahora fijamos un número primo $p$ que escinde en $K$ y no es divisible por primos de $S$.

Como se explica en \cite{Serre}, por el método de Demjanenko--Manin existe un entero $n=n(p,K)\ge 1$ tal que el conjunto de puntos $K$-racionales de la curva modular $X_0(p^n)$ es finito. Así, salvo finitas clases de $K^{\rm alg}$-isomorfismo (que descartamos desde ya, sin afectar el resultado deseado), las curvas elípticas sobre $K$ tienen su $G_K$-representación $p$-ádica de Tate $V_p(E) = T_p(E)\otimes\Q_p$ irreducible.

El lemma de comparación de Faltings (originalmente implícito en el Teorema 5 de \cite{Faltings} y explicitado en \cite{Deligne}) nos dice que cuando $E$ varía sobre nuestras curvas elípticas con buena reducción fuera de $S$ y con $V_p(E)$ irreducible (así, semisimple), esta $G_K$-representación sólo alcanza finitas clases de isomorfismo. Aquí basta usar la cota de Hasse para controlar las trazas de Frobenius. 

Una $G_K$-representación $\Q_p$-lineal irreducible de dimensión $2$ sólo contiene finitos $\Z_p$-retículos $G_K$-estables salvo homotecia. Esto viene de consideraciones elementales sobre el árbol de Bruhat--Tits, ver por ejemplo la Proposición 2.3.12 de \cite{Argaez}. Así, cuando $E$ varía sobre nuestras curvas elípticas, los módulos de Tate $T_p(E)$ sólo alcanzan finitas clases de isomorfismo. Sólo falta mostrar que cada una de estas clases proviene de a lo más finitas curvas elípticas sobre $K$, salvo $K^{\rm alg}$-isomorfismo. 

Recordando cómo elegimos $p$ y buscando la finitud de los invariantes $j$ relevantes, lo requerido viene del siguiente resultado local: \emph{dada una $G_{\Q_p}$-representación $T$ sobre $\Z_p$ de rango $2$, sólo existen finitas curvas elípticas $C$ sobre $\Q_p$ con buena reducción, salvo $\Q_p$-isomorfismo, tales que su módulo de Tate $T_p (C)$ es isomorfo a $T$ como $G_{\Q_p}$-representación $\Z_p$-lineal}. En efecto, si $\Ccal$ es el modelo de Néron de $C$ sobre $\Z_p$ con fibra especial $C_0$ (la cual es curva elíptica), el Corolario 1 al Teorema 4 de \cite{Tate} muestra que $T_p(C)$ determina de forma única el grupo $p$-divisible $\Ccal[p^\infty]$, y por el teorema de Serre--Tate (ver la sección 6 del reporte \cite{LubinSerreTate}) este grupo $p$-divisible en conjunto con $C_0$ determinan de forma única la deformación $\Ccal$ de $C_0$ (y así, de forma única a $C$). Sólo resta observar que el número de clases de isomorfismo de curvas elípticas $C_0$ sobre $\F_p$ es finito.
\end{proof}

\subsection*{Sobre los resultados de finitud aritmética utilizados} Demjanenko--Manin utiliza propiedades básicas de alturas y el teorema de Mordell--Weil. El lema de Faltings usa el teorema de Hermite--Minkowski, el teorema de Chebotarev, el lema de Nakayama y el teorema de Brauer--Nesbitt y, aceptando estos resultados, el argumento es bastante corto (ver p.249 de \cite{Deligne}).

\subsection*{Discusión sobre otras demostraciones} Clásicamente el Teorema \ref{TeoShafarevich} viene del teorema de Siegel sobre puntos enteros en curvas aplicado a una curva elíptica adecuada o, alternativamente, del teorema de Siegel de finitud de soluciones de la ecuación $S$-unidad (aunque aquí es necesario pasar a una extensión de $K$). Ambas estrategias nacen en la aproximación diofantina (Thue, Siegel, Roth, ...) y una alternativa moderna es la teoría de formas lineales en logaritmos. 

Desde otro frente, el Teorema \ref{TeoShafarevich} es el caso de dimensión $1$ de la conjetura de Shafarevich demostrada por Faltings en  \cite{Faltings}; de hecho, nuestro argumento se inspira hasta cierto punto en el método de Faltings pero en realidad es distinto. Por la vía de la ecuación $S$-unidad sobre campos de números, se obtiene otra demostración libre de aproximación diofantina gracias a los argumentos de Lawrence y Venkatesh en \cite{LawrenceVenkatesh} usando mapas de periodos $p$-ádicos. Mencionamos que otras demostraciones de finitud de la ecuación $S$-unidad existen sobre $\Q$ tales como \cite{Kim, MurtyPasten, Poonen} entre otras, pero no parecen ser suficientes para obtener el resultado sobre campos de números en general.

\subsection*{Aplicaciones} En la dirección recíproca, el Teorema \ref{TeoShafarevich} aplicado a la familia de Legendre da como consecuencia formal la finitud de soluciones de la ecuación $S$-unidad sobre campos de números (y así, muchos casos del Teorema de Siegel para puntos enteros en curvas) por lo que nuestro argumento da un nuevo punto de entrada a estos resultados fundamentales de la teoría de ecuaciones diofantinas.


\section*{Agradecimientos}

Esta investigación fue financiada por ANID Fondecyt Regular 1230507 de Chile. Agradezco Yuri Bilu y a Anthony Várilly-Alvarado por comentarios en una primera versión, y a Siddharth Mathur por aclarar unas dudas.


\end{document}